\documentclass[11pt, sort&compress]{elsarticle} 

\usepackage{tikz}
\usetikzlibrary{arrows}

\usepackage{fullpage}
\usepackage{amsfonts} 
\usepackage{graphics}

\usepackage{url}      

\usepackage{graphicx}
\usepackage{amsmath, amsthm, amssymb}
\usepackage{amssymb,amsmath,graphics,color}

\usepackage{epsfig}
\usepackage{epstopdf}

\usepackage{multirow}

\usepackage{tocloft}

\usepackage{xr} 

\usepackage{xcolor}

\definecolor{darkbrown}{rgb}{.5,.1,.1} 

\usepackage{here} 

\usepackage{relsize} 

\usepackage{refcheck} 
\norefnames	
\nocitenames 	
\ignoreunlbld 
\setoffmsgs 

\renewcommand{\today}
{%
\number \day 
\ifnum  \day =1 st
\else
	\ifnum  \day =2 nd
	\else
		\ifnum  \day =3 rd
		\else
			\ifnum  \day =21 st
			\else
				\ifnum  \day =22 nd
				\else
					\ifnum  \day =23 rd
					\else
						\ifnum  \day =31 st
						\else
							th
						\fi
					\fi
				\fi
			\fi
		\fi
	\fi
\fi
\ifcase \month \or January\or February\or March\or April\or May \or June\or July\or August\or September\or October\or November\or December\fi 
{ \number \year}%
}

\def\boxit [#1]#2{\vbox{\hrule\hbox{\vrule
     \vbox spread #1{\vss\hbox spread#1{\hss #2\hss}\vss}%
        \vrule}\hrule}}

\newbox\algo

\font\algofont=cmtt10 scaled 1100

\newenvironment {algorithme}{\smallskip \smallskip
\bgroup\algofont\parindent= 0mm}{\egroup
\smallskip\smallskip}

\newtheorem{thm}{Theorem}[section]
\newtheorem{cor}[thm]{Corollary} 
\newtheorem{prop}[thm]{Proposition} 
\newtheorem{lemma}[thm]{Lemma}
\newtheorem{remark}[thm]{Remark}
\newtheorem{definition}[thm]{Definition}
\newtheorem{observation}[thm]{Observation}
\newtheorem{example}[thm]{Example}

\newtheorem{hors-property}[subsubsection]{Property}
\newtheorem{hors-lemma}[subsubsection]{Lemma}
\newtheorem{hors-remark}[subsubsection]{Remark}
\newtheorem{hors-definition}[subsubsection]{Definition}

\def\newline{\hfill\break}

\def\newpage{\vfill\break}

\def\min{\mathrm{min}}

\def\max{\mathrm{max}}

\def\n{\omega}

\def\ov{\overrightarrow}
\def\s{\setminus}
\def\bk{\backslash}

\def\ms{\medskip}
\def\ss{\smallskip}

\def\n{\omega}

\def\bul{$\bullet$ }

\def\G{{\ov G}}

\def\C{C_{\hbox{\small opt}}}

\def\bysame{\rule{1cm}{0.1mm}}

\newcommand{\overbar}[1]{\mkern 1.5mu\overline{\mkern-1.5mu#1\mkern-1.5mu}\mkern 1.5mu}
\def\bar{\overbar}    

\def\eme#1{}
\def\emeref#1{}
\def\emevder#1{}

\title{Computing the fully optimal spanning tree \\ of an ordered bipolar directed graph}

\author[lirmm]{Emeric Gioan\corref{cor1}}
\ead{emeric.gioan@lirmm.fr}

\author[paris]{Michel Las Vergnas\corref{cor2}}

\cortext[cor1]{Email: emeric.gioan@lirmm.fr}
\cortext[cor2]{R.I.P.}
\address[lirmm]{CNRS, LIRMM, Universit\'e de Montpellier,  France}
\address[paris]{CNRS, Paris, France}

\date{}

\begin{document}

\begin{abstract} 
\noindent
It has been previously shown by the authors that a directed graph on a linearly ordered set of edges (ordered graph) with adjacent unique source and sink (bipolar digraph) has a unique 
fully optimal spanning tree, that satisfies a simple criterion on fundamental cycle/cocycle directions.
This result yields, for any ordered graph, a canonical bijection between bipolar orientations and spanning trees with internal activity 1 and external activity 0 in the sense of the Tutte polynomial. This bijection can be extended to  all orientations and all spanning trees, yielding the active bijection, presented for graphs in a companion paper.
In this paper, we specifically address the problem of the computation of the fully optimal spanning tree of an ordered bipolar digraph. In contrast with the inverse mapping, built by a straightforward single pass over the edge set, the direct computation is not easy and had previously been left aside.
We give two independent constructions. The first one is a deletion/contraction recursion, involving an exponential number of minors. 
\emevder{It is structurally significant but }%
It is structurally significant but it is efficient only for building the whole 
bijection (i.e. all images) at once. The second one is more complicated and is the main contribution of the paper. It involves 
just one minor for each edge of the resulting spanning tree,
and it is a translation and an adaptation 
in the case of graphs, 
in terms of weighted cocycles,
of a general geometrical 
linear programming type 
algorithm, 
which allows for a polynomial time~complexity.
\emevder{mettre which yields a polynomial time~complexity ?}
\emevder{parler de wieght de cocycles ?}
\end{abstract}

\maketitle   


\emevder{faire corrections d'anglais, comme dans chapter, voir fichier anglais.tex  passage "SUR CHAPTER"}%

\section{Introduction}
\label{sec:intro}

In a previous paper \cite{GiLV05} (see also \cite{AB1}), we showed that a directed graph $\G$ on a linearly ordered set of edges,  with adjacent unique source and sink connected by the smallest edge (ordered bipolar digraph), has a unique remarkable spanning tree that satisfies a simple criterion on fundamental cycle/cocycle directions, and that we call \emph{the fully optimal spanning tree $\alpha(\G)$ of $\G$}.
 Associating bipolar orientations of an ordered graph with their fully optimal spanning trees provides a canonical bijection with spanning trees with internal activity 1 and external activity 0 in the sense of the Tutte polynomial \cite{Tu54}. It is a classical result from \cite{Za75} that those two sets have the same size,  also known as the $\beta$-invariant $\beta(G)$ of the graph \cite{Cr67}.  
We call this bijection \emph{the uniactive bijection of $G$}.
 
 This bijection can be extended to all orientations and all spanning trees, yielding 
\emph{the active bijection}, introduced in terms of graphs first in \cite{GiLV05}, and detailed next in \cite{ABG2}, for which the present paper is a complementary companion paper.
 Beyond graphs, the general context of the active bijection is oriented matroids, as  studied by the authors in a series of papers, even more detailed, notably~\cite{AB1, AB2-b, AB3, AB4}.\break
 See the introductions of \cite{ABG2} (or also \cite{AB2-b}) for an overview, more information and references.
 The general purpose of these works is to study graphs (or hyperplane arrangements, or oriented matroids) on a linearly ordered set of edges (or ground set), under 
 various 
 structural and enumerative aspects.
%




In the present paper we address the problem of computing the fully optimal spanning tree. 
Its existence and uniqueness is a tricky combinatorial theorem, proved in \cite[Theorem 4]{GiLV05} (see also the companion paper \cite[Section \ref{ABG2-subsec:fob}]{ABG2} for a short sum up in graphs, see also \cite[Theorem 4.5]{AB1} for a generalization and a geometrical interpretation in oriented matroids, see also \cite[Section 5]{AB2-b}\emevder{verifier ref} for a summary of various interpretations and implications of this theorem). 
As recalled in Section \ref{subsec:prelim-fob}, the inverse mapping, producing a bipolar orientation for which a given spanning tree is fully optimal, is very easy to compute by a single pass over the ordered set of edges. 
But the direct computation is complicated and it had not been addressed in previous papers. When generalized to real hyperplane arrangements, the problem contains and strengthens the real linear programming problem (as shown in~\cite{AB1}, hence the name \emph{fully optimal}). 
%
This ``one way function'' feature is a noteworthy aspect of the active bijection.
Here, we give two independent constructions to compute the mapping, that is to compute the fully optimal spanning tree of an ordered bipolar digraph.


The first construction,  in Section \ref{sec:induction}, is recursive, by deletion/contraction of the greatest element. 
Let us observe that it is usual to have some deletion/contraction constructions when the Tutte polynomial is involved, and that this construction fits a general framework involving all orientations and spanning trees, presented in \cite[Section \ref{ABG2-subsec:ind-framework}]{ABG2}\emevder{ou juste citer section  \cite[Section \ref{ABG2-sec:induction}]{ABG2} pas sousection} and detailed in \cite{AB4}.
This construction of the mapping has a short statement and proof, and it can be used to efficiently build the whole bijection at once (i.e. all the images simultaneously, see Remark \ref{rk:ind-10}).
So, it is satisfying for the structural understanding and for a global approach,
but it is not satisfying in terms of computational complexity for building one single image
as it involves an exponential number of minors. 
%

%

The second construction, in Section \ref{sec:fob},  is more technical and is the main contribution of the paper. It is 
 efficient from the computational complexity viewpoint because it involves 
 only one minor for each edge of the resulting spanning tree,
 and it consists in searching, successively in each minor, for the smallest cocycle with respect to a linear ordering of the set of cocycles induced by a suitable weight function.
This algorithm is an adaptation in the case of graphs (implicitly using that graphic matroids are binary) of a general geometrical construction 
obtained by elaborations on
pseudo/real linear programming (in oriented matroids / real hyperplane arrangements). Briefly, the ordering of cocycles
is a substitue for some
multiobjective programming, where a vertex (geometrical counterpart of a cocycle) is optimized with respect to a sequence of objective functions (transformed here into weights on edges).
%
By this way, the fully optimal spanning tree can be computed in polynomial time.
See  Remark \ref{rk:lp} and the end of Section \ref{sec:fob} for more discussion. See  \cite{AB3} for the general geometrical construction.
See also \cite{GiLV09} for a short formulation of the same algorithm in terms of real hyperplane arrangements. 
See \cite{AB1} for the primary relations between the full optimality criterion and usual linear programming optimality (see also \cite{GiLV04} in the uniform case).


In addition, 
let us recall from \cite[Section 4]{GiLV05}  (see also \cite[Section \ref{ABG2-subsec:fob}]{ABG2} for more details) that the bijection between bipolar orientations and their fully optimal spanning trees  directly yields a  bijection between cyclic-bipolar orientations (the strongly connected orientations obtained from bipolar orientations by reversing 
the source-sink edge) and spanning trees with internal activity $0$ and external activity~$1$.
Hence, the algorithms developed here can also be used for this second bijection.
Let us mention that 
 this framework
involves a remarkable duality property, 
called \emph{the active duality},
which is reflected in several ways. 
First, those two bijections are related to each other consistently with cycle/cocycle duality (that is oriented matroid duality, which extends planar graph duality, see \cite[Section \ref{ABG2-subsec:fob}]{ABG2} in graphs, see also \cite[Section 5]{AB1}, or \cite[Section 5]{AB2-b}\emevder{----verifier ref----} for a complete overview). Second, this duality property can be seen as a strengthening of linear programming duality (see \cite[Section 5]{AB1}). 
Third, it is related to the equivalence of  two dual  formulations in the deletion/contraction construction (see Remark~\ref{rk:ind-10-equivalence}).


Lastly, it is important to  point out that the two aforementioned constructions of the fully optimal spanning tree do not give a new proof of its existence and uniqueness:
on the contrary, 
this crucial fundamental result 
is used
to ensure the correctness of these two constructions.


%
%


\section{Preliminaries}
\label{sec:prelim}

\subsection{Usual terminology and tools from oriented matroid theory}
\label{subsec:prelim-tools}

\emevder{remplacer le plus possible par ordered graph ?}%

All graphs considered in this paper will be connected.
They can have loops and multiple edges. 
%
%
A \emph{digraph} is a directed graph, and an \emph{ordered graph} is a graph $G=(V,E)$ on a linearly ordered set of edges $E$. 
Edges of a directed graph are supposed to be \emph{directed} or equally \emph{oriented}.
A directed graph will be denoted with an arrow, $\G$, and the underlying undirected graph without arrow, $G$. 
%
The cycles,  cocycles, and spanning trees of a graph $G=(V,E)$ are considered as subsets of $E$,
hence their edges can be called their \emph{elements}. 
The cycles and cocycles of $G$ are always understood as being minimal for inclusion.
Given $F\subseteq E$, we denote $G(F)$ the graph obtained by restricting the edge set of $G$ to $F$, that is the minor $G\s (E\s F)$ of $G$.
%
A minor $\G/\{e\}$, resp. $\G\bk\{e\}$, for $e\in E$, can be denoted for short $\G/e$, resp. $\G\bk e$.
For $e\in E$, we denote $-_e\G$ the digraph obtained by reversing the direction of the edge $e$ in $\G$.
%
%

Let $G$ be an ordered (connected) graph and let $T$ be  a spanning tree of $G$. 
For $t\in T$, the \emph{fundamental cocycle} of $t$ with respect to $T$, denoted $C_G^*(T;t)$, or $C^*(T;t)$ for short,  is the cocycle joining the two connected components of $T\setminus \{t\}$. Equivalently, it is the unique cocycle contained in $(E\s T)\cup\{t\}$.
For $e\not\in T$, the \emph{fundamental cycle} of $e$ with respect to $T$, denoted $C_G(T;e)$, or $C(T;e)$ for short,
 is the unique cycle contained in $T\cup\{e\}$.

\medskip


The technique used in the paper is close from oriented matroid technique, which notably means that it focuses on edges, whereas vertices are usually not used.
%
%
Given an orientation $\G$ of a graph $G$, we will have to deal with directions of edges in cycles and cocycles of the underlying graph $G$,  and, sometimes,  to deal with combinations of cycles or cocycles. To achieve this, it is  convenient to use some practical notations and classical properties from oriented matroid theory~\cite{OM99}.
%

A \emph{signed edge subset} is a subset $C\subseteq E$ provided with a partition  into a positive part $C^+$ and a negative part $C^-$.
A cycle, resp. cocycle, of $G$ provides two opposite signed edge subsets called \emph{signed cycles}, resp. \emph{signed cocycles}, of $\G$ by giving a sign in $\{+,-\}$ to each of its elements accordingly with the orientation $\G$ of $G$ the natural way.  
Precisely: two edges having the same direction with respect to a running direction of a cycle will have the same sign in the associated signed cycles, and  two edges having the same direction with respect to the partition of the vertex set induced by a cocycle will have the same sign in the associated signed cocycles.
%
In particular, a directed cycle, resp. a directed cocycle, of $\G$ corresponds to a signed cycle, resp. a signed cocycle, all the elements of which are positive (and to its opposite, all the elements of which are negative).
We will often use the same notation $C$ either for a signed edge subset (formally a couple $(C^+,C^-)$, e.g. signed cycle) or for the underlying subset ($C^+\uplus C^-$, e.g. graph cycle).
%
\emevder{eventuellement faire en sorte d'enlever ce qui suit, mais bon sinon laisser car ca fait chier de nettoyer et on s'en sert dans autres papiers...}%
When necessary, given a spanning tree $T$ of $G$ and an edge $t\in T$, resp. an edge $e\not\in T$, the fundamental cocycle $C^*(T;t)$, resp. the  fundamental cycle $C(T;e)$, induces two opposite signed cocycles, resp. signed cycles, of $\G$; then, by convention, 
the (signed) fundamental cocycle $C^*(T;t)$, resp. the (signed) fundamental cycle $C(T;e)$, is considered to be the one in which $t$ is positive, resp. $e$ is positive.

\bigskip

The next three tools can be skipped in a first reading, as they will only be used in the proof of the main result of the paper, namely Theorem \ref{th:fob}.
First, let us  recall the definition of the \emph{composition} $C\circ D$ between two signed edge subsets as the edge subset $C\cup D$ with signs inherited from $C$ for the element of $C$ and inherited from $D$ for the elements of $D\setminus C$.
We will use the classical {\it orthogonality} property between a cocycle $D$ and a composition of cycles $C$ of $\G$, that is: $C\cap 
D\not=\emptyset$ implies
$(C^+\cap D^+)\cup(C^-\cap D^-)\not=\emptyset$ and $(C^-\cap D^+)\cup(C^+\cap 
D^-)\not=\emptyset$.

Second, we recall that, given two cocycles $C$ and $C'$ of $\G$,
and an element $f \in C \cup C'$ which does not have opposite signs in $C$ and $C'$,
there exists a cocycle $D$ obtained by \emph{elimination between $C$ and $C'$
preserving $f$} such that  $f\in D$, 
$D^+ \subseteq  C^+\cup C'^+$,
$D^- \subseteq  C^-\cup C'^-$,
and $D$ contains no element of $C\cap C'$ having opposite signs in $C$ and $C'$. This last property is a strengthening of the oriented matroid elimination property in the particular case of digraphs, 
a short proof of which is the following.
Assume $C$ defines the partition $(C_1,C_2)$ of the set of vertices, and
$C'$ defines the partition $(C'_1,C'_2)$, with a positive sign given to edges from $C_1$ to $C_2$ in $C$ and from $C'_1$ to $C'_2$ in $C'$. Then the edges having opposite signs in $C$ and $C'$ are those joining $C_1\cap C'_2$ and $C_2\cap C'_1$, then, with $V'= (C_1\cap C'_2) \cup (C_2\cap C'_1)$, the cut defined by the partition $(V',E\setminus V')$ contains a cocycle answering the problem.
\emevder{faire lemme de cette elimination speciale ?}%

Third, we recall the following easy property.
Let $A,B\subseteq E$ with $A\cap B=\emptyset$, such that the minor $G/B\backslash A$ is connected (or equivalently: $G/B\backslash A$ has the same rank as $G/B$).
If $D'$ is a cocycle of $G/B\backslash A$, then there exists a unique 
cocycle $D$ of $G$ such that $D\cap B=\emptyset$ and $D\setminus A=D'$. 
If the graphs are directedd then $D$ has the same signs as $D'$ on the elements of $D'$. We say that $D'$ is \emph{induced by} $D$, or that $D$ \emph{induces} $D'$. 

\subsection{Bipolar orientations and fully optimal spanning trees}
\label{subsec:prelim-fob}

We say that a directed graph $\G$  on the edge set $E$ is {\it bipolar with respect to $p\in E$} if  $\G$ is acyclic and has a unique source and a unique sink which are the extremities of $p$.
In particular, if $\G$ consists in a single edge $p$ which is an isthmus, then $\G$ is bipolar with respect to $p$.
Equivalently, $\G$ is bipolar with respect to $p$ if and only if every edge of $\G$ is contained in a directed cocycle and every directed cocycle contains $p$ 
(for information: in other words, $\G$ has dual-orientation-activity $1$ and orientation-activity $0$, in the sense of \cite{LV84a}, see  \cite{GiLV05, AB2-b} or \cite[Section \ref{ABG2-subsec:prelim-beta}]{ABG2}).
%
%
Another characterization is the following: $\G$ is bipolar w.r.t. $p$  if and only if $\G$ is acyclic and $-_p\G$ is strongly connected (for information:  those orientations $-_p\G$ play an equivalent dual role, see the discussion at the end of Section \ref{sec:intro} or \cite[Section \ref{ABG2-subsec:fob}]{ABG2}).
%
%

%
%


\begin{definition}
\label{def:acyc-alpha}
Let $\G=(V,E)$ be a directed graph, on a linearly ordered set of edges, which is bipolar 
with respect to the minimal element $p$ of $E$. The {\it fully optimal spanning tree} $\alpha(\G)$ of $\G$ is the unique spanning tree $T$ of $G$ such~that:
\smallskip

\bul for all $b\in T\setminus p$, the directions (or the signs) of $b$ and $\min(C^*(T;b))$ are opposite in $C^*(T;b)$;

\bul for all $e\in E\s T$, the directions (or the signs) of $e$ and $\min(C(T;e))$ are opposite in $C(T;e)$.
\end{definition}

The existence and uniqueness of such a spanning tree is the main result of \cite{GiLV05, AB1}, along with the next theorem. 
Notice that a directed graph and its opposite are mapped onto the same spanning tree.
We say that spanning tree $T$ has \emph{internal activity $1$ and external activity $0$}, or equivalently that $T$ is \emph{uniactive internal}, if: $min(E)\in T$; for every $t\in T\s min(E)$ we have $t\not=\min(C^*(T;t))$; and for every $e\in E\s T$ we have $e\not=\min(C(T;e))$. 

In this paper, it is not necessary to define further the notion of activities of spanning trees, which comes from the theory of the Tutte polynomial (see \cite{GiLV05, ABG2, AB2-b}). 
\emevder{verifier que activities pas utilisees}%
%
%
%
For information  (not used in the paper), the number of uniactive internal spanning trees of $G$ does not depend on the linear ordering of $E$ and is known as the $\beta$-invariant $\beta(G)$ of $G$~\cite{Cr67}, while the number of bipolar orientations w.r.t. $p$
does not depend on $p$ and is equal to $2.\beta(G)$ \cite{Za75}.
\emevder{OU : while the number of bipolar orientations w.r.t. $p$, with the same fixed orientation for $p$, does not depend on $p$ and is also equal to $\beta(G)$ \cite{Za75}.}%
\emevder{faut-il laisser ce "for information..."?}%


%
%

\begin{thm}[Key Theorem \cite{GiLV05, AB1}]
\label{thm:bij-10}
Let $G$ be a graph on a linearly ordered set of edges $E$ with $\min(E)=p$.
The mapping 
$\G\mapsto \alpha(\G)$ yields
a bijection between all bipolar orientations of $G$ w.r.t.~$p$, with the same  fixed orientation for $p$, and all 
uniactive internal spanning trees of $G$.
\end{thm}

The bijection of Theorem \ref{thm:bij-10} is called \emph{the uniactive bijection} 
of the ordered graph $G$.

For completeness of the paper (though not used thereafter), let us recall that,
from the constructive viewpoint, this bijection was built in \cite{GiLV05,AB1} by the  inverse mapping, provided  by a single pass algorithm over $T$ and fundamental cocycles, or dually over $E\setminus T$ and fundamental cycles.
This algorithm is illustrated in \cite[Figure 1]{GiLV05}, on the same example that we will use in Section~\ref{sec:fob}.
Equivalently, 
the inverse mapping can be obviously built by  a single pass over $E$, choosing edge directions one by one so that the criterion 
of Definition \ref{def:acyc-alpha} is satisfied. 
We recall this algorithm below 
(as done also in \cite{AB1} and \cite[Section \ref{ABG2-subsec:basori}]{ABG2}).
\emevder{pas utile de repeter cet algo ici... on peut juste remarquer qu'il suffit de signer elemetns un par un et renvoyer a ABG2... en meme temps ca complete... laisser ou pas ?}%
The reader interested in a geometric intuition on the full optimality sign criterion can have a look at the equivalent definitions, illustrations and interpretations given in \cite{AB1, AB2-b, AB3}.
\emevder{utile de dire ce geometric intutiion? reference a remark LP ?}

\begin{prop}[{self-dual reformulation of \cite[Proposition 3]{GiLV05}}]
\label{prop:alpha-10-inverse}
Let $G$ be a graph on  a linearly ordered set of edges $E=\{e_1,\dots,e_n\}_<$. 
For a uniactive internal spanning tree $T$ of $G$,
the two opposite orientations of $G$ whose image under $\alpha$ is $T$ are computed by the following algorithm.

\begin{algorithme}
Orient $e_1$ arbitrarily.\par
For $k$ from $2$ to $n$ do\par
\hskip 10 mm if $e_k\in T$ then \par
\hskip 20 mm let $a=\min (C^*(T;e_k))$\par
\hskip 20 mm orient $e_k$ in order to have $a$ and $e_k$ with opposite directions in $C^*(T;e_k)$\par
\hskip 10 mm if $e_k\not\in T$ then \par
\hskip 20 mm let $a=\min (C(T;e_k))$\par
\hskip 20 mm orient $e_k$ in order to have $a$ and $e_k$ with opposite directions in $C(T;e_k)$\par
\end{algorithme}

\end{prop}

Lastly, in the proof of the main result Theorem \ref{th:fob}, we will use the 
 following \emph{alternative characterization of the fully optimal spanning tree}, equivalent to Definition \ref{def:acyc-alpha} by \cite[Proposition 3]{GiLV05} or by \cite[Proposition 3.3]{AB1}.
Let $\G=(V,E)$ be an ordered directed graph, which is bipolar 
with respect to $p=\min(E)$.
Let $T=\alpha(\G)$.
With the convention that an edge has a positive sign in its fundamental cycle or cocycle w.r.t. a spanning tree, 
with $T=b_1<b_2<...<b_r$,
and with $E\setminus T=\{c_1<...<c_{n-r}\}$,
%
we have:
%
\vspace{-1mm}
\begin{itemize}
\itemsep=0mm
\partopsep=0mm 
\topsep=0mm 
\parsep=0mm
\item $b_1=p=\min(E)$;
\item
for every $1\leq i\leq r$, all elements of 
$\cup_{j\leq i} C^*(T;b_j)\setminus \cup_{j\leq i-1} C^*(T;b_j)$ are positive in $C^*(T;b_i)$; 
\item 
for every $1\leq i\leq n-r$, all elements of $\cup_{j\leq i} C(T;c_j)\setminus \cup_{j\leq i-1} C(T;c_j)$ are positive in $C(T;c_i)$ except $p=\min(E)$.
\end{itemize}
Equivalently, as formulated in \cite{AB1}, in terms of compositions of signed subsets (see Section \ref{subsec:prelim-tools}), the two latter properties can be formulated the following way: $C^*(T;b_1)\circ ...\circ C^*(T;b_r)$ is positive; and $C(T;c_1)\circ ...\circ C(T;c_{n-r})$ is positive except on $p$.
%


%


\emevder{rk (deja pris en compte nprmalement) : elemetns de T sont b dans prelim et dans fob dans preuve de algo final, mais  t au cours de algo final... rk: $t$ est deja polynome de Tutte dans abg2 mais pas ici, donc pas de pb}%


%
\vspace{-2mm}
\section{Recursive construction by deletion/contraction}
\label{sec:induction}

\vspace{-2mm}


This section investigates
a recursive deletion/contraction construction of the fully optimal spanning tree of an ordered bipolar digraph. 
See Section \ref{sec:intro} for an outline.
This construction
 is developed further in \cite[Section \ref{ABG2-sec:induction}]{ABG2} by giving deletion/contraction constructions involving all orientations and spanning trees%
\footnote{Note: Theorem \ref{thm:ind-10} is also stated in the companion paper \cite{ABG2}, which is also submitted. At the moment, we give its proof in both papers, including Lemma \ref{lem:induc-fob-basis}, but we should eventually remove this repetition and give the proof in only one of the two papers.}%
, and even more further in \cite{AB4} by generalizing such constructions in oriented matroids.
Let us mention that the construction, as it is formally stated below, extends directly to compute the fully otpimal basis of a bounded region of an  oriented matroid, as addressed in \cite{AB1}.



%
%
%


\begin{lemma}
\label{lem:induc-fob-basis}
Let $\G$ be a digraph,  on a linearly ordered set of edges $E$, which is bipolar w.r.t. $p=\min(E)$. Let $\n$ be the greatest element of $E$. Let $T=\alpha(\G)$. 
If $\n\in T$ then $\G/\n$ is bipolar w.r.t. $p$ and $T\setminus\{\n\}=\alpha(\G/\n)$.
If $\n\not\in T$ then $\G\backslash\n$ is bipolar w.r.t. $p$ and $T=\alpha(\G\backslash\n)$.
In particular, we get that $\G/\n$ is bipolar w.r.t. $p$ or $\G\backslash\n$ is bipolar w.r.t. $p$.
\end{lemma}

\begin{proof}
First, let us recall that if a spanning tree of a directed graph satisfies the criterion of Definition \ref{def:acyc-alpha}, then this directed graph  is necessarily bipolar w.r.t.its smallest edge. 
This is implied by \cite[Propositions 2 and 3]{GiLV05}, or also stated explicitly in \cite[Proposition 3.2]{AB1},
and this is easy to see: if the criterion is satisfied, then 
the spanning tree is internal uniactive (by definitions of internal/external activities)
and the digraph is determined up to reversing all edges (see Proposition \ref{prop:alpha-10-inverse}\emevder{ATTENTION ptet pas dans ABG2 ???}), 
which implies that the digraph is in the inverse image  of $T$ by the uniactive bijection of Theorem \ref{thm:bij-10} and that it is bipolar w.r.t. its smallest edge.

Assume that $\n\in T$.
Obviously, the fundamental cocycle of $b\in T\s\{\n\}$ w.r.t. $T\s\{\n\}$ in $G/\n$ is the same as the fundamental cocycle of $b$ w.r.t. $T$ in $G$.
And the fundamental cycle of $e\not\in T$ w.r.t. $T\s\{\n\}$ in $G/\n$ is obtained by removing $\n$ from the fundamental cycle of $e$ w.r.t. $T$ in $G$.
Hence, those fundamental cycles and cocycles in $G/\n$ satisfy the criterion of Definition \ref{def:acyc-alpha}, hence $\G/\n$ is bipolar w.r.t. $p$ and $T\setminus\{\n\}=\alpha(\G/\n)$.

Similarly (dually in fact), assume that $\n\not\in T$.
%
%
The fundamental cocycle of $b\in T$ w.r.t. $T\s\{\n\}$ in $G\bk\n$ is obtained by removing $\n$ from the fundamental cocycle of $b$ w.r.t. $T$ in $G$.
And the fundamental cycle of $e\not\in T\s\{\n\}$ w.r.t. $T\s\{\n\}$ in $G\bk\n$ is the same as the fundamental cycle of $e$ w.r.t. $T$ in $G$.
Hence, those fundamental cycles and cocycles in $G\bk\n$ satisfy the criterion of Definition \ref{def:acyc-alpha}, hence $\G\bk\n$ is bipolar w.r.t. $p$ and $T\setminus\{\n\}=\alpha(\G\bk\n)$.

Note that the fact that either $\G/\n$ is bipolar w.r.t. $p$, or $\G\backslash\n$ is bipolar w.r.t. $p$ could also easily be directly proved in terms of digraph properties.
\end{proof}

\begin{thm}
\label{thm:ind-10}

Let $\G$ be a digraph,  on a linearly ordered set of edges $E$, which is bipolar w.r.t. $p=\min(E)$. 
The fully optimal spanning tree $\alpha(\G)$ of $\G$ satisfies 
the following inductive definition, 
where $\n=\max(E)$.%

\emevder{attention newpage dessous}%
\newpage

\begin{algorithme}


If $|E|=1$ then $\alpha(\G)=\n$.\par
If $|E|>1$  then:\par


\hskip 10mm If $\G/\n$ is bipolar w.r.t. $p$ but not $\G\backslash \n$  then
$\alpha(\G)=\alpha(\G/\n)\cup\{\n\}$.

\hskip 10mm If $\G\backslash\n$ is bipolar w.r.t. $p$ but not $\G/ \n$ then
$\alpha(\G)=\alpha(\G\backslash\n)$.

\hskip 10mm If both $\G\backslash\n$ and $\G/\n$ are bipolar w.r.t. $p$  then:


\hskip 20mm
let $T'=\alpha(\G\backslash\n)$, $C=C_\G(T';\n)$ and $e=\min(C)$\par

\hskip 20mm
if $e$ and $\n$ have opposite directions in $C$ then $\alpha(\G)=\alpha(\G\backslash\n)$;\par
\hskip 20mm
if $e$ and $\n$ have the same directions in $C$ then $\alpha(\G)=\alpha(\G/\n)\cup\{\n\}$.\par
\smallskip

{\sl or equivalently:}\par
\hskip 20mm
let $T''=\alpha(\G/\n)$, $D=C^*_\G(T''\cup\n;\n)$ and $e=\min(D)$\par

\hskip 20mm
if $e$ and $\n$ have opposite directions in $D$ then $\alpha(\G)=\alpha(\G/\n)\cup\{\n\}$;\par
\hskip 20mm
if $e$ and $\n$ have the same directions in $D$ then $\alpha(\G)=\alpha(\G\backslash\n)$.
\smallskip

\end{algorithme}
\end{thm}


%

\begin{proof}
By Lemma \ref{lem:induc-fob-basis}, at least one minor among $\{\G/\n, \G\bk\n\}$ is bipolar w.r.t. $p$.
If exactly one minor among $\{\G/\n, \G\bk\n\}$ is bipolar w.r.t. $p$, then
by Lemma \ref{lem:induc-fob-basis} again, 
the above definition  is implied.  Assume now that both minors are bipolar w.r.t. $p$.

Consider $T'=\alpha(\G\bk\n)$. 
Fundamental cocycles of elements in $T'$ w.r.t. $T'$ in $\G$ are obtained by removing $\n$ from those in $\G\bk\n$. Hence they satisfy the criterion from Definition \ref{def:acyc-alpha}.
Fundamental cycles of elements in $E\s (T'\cup\{\n\})$ w.r.t. $T'$ in $\G$ are the same as in $\G\bk\n$.  Hence they satisfy the criterion from Definition \ref{def:acyc-alpha}.
Let $C$ be the fundamental cycle of $\n$ w.r.t. $T'$. 
If $e$ and $\n$ have opposite directions in $C$,
then $C$ satisfies the criterion from Definition \ref{def:acyc-alpha}, and
 $\alpha(\G)=T'$.
Otherwise, we have $\alpha(\G)\not=T'$, and, by Lemma \ref{lem:induc-fob-basis}, we must have $\alpha(\G)=\alpha(\G/\n)\cup\{\n\}$.

The second condition involving $T''=\alpha(\G/\n)$ is proved in the same manner. Since it yields the same mapping $\alpha$, then this second condition is actually equivalent to the first one, and so it can be used as an alternative. Note that the fact that these two conditions are equivalent is difficult and proved here in an indirect way (actually this fact is equivalent to the key result that $\alpha$ yields a bijection), see Remark \ref{rk:ind-10-equivalence}. 
\end{proof}

\begin{cor}
\label{cor:ind-10}
Using notations of Theorem \ref{thm:ind-10}, if $-_\n\G$ is bipolar w.r.t. $p$ then the above algorithm of Theorem \ref{thm:ind-10} builds at the same time $\alpha(\G)$ and $\alpha(-_\n\G)$, 
we have:
$$\Bigl\{\ \alpha(\G),\ \alpha(-_\n\G)\ \Bigr\}\ =\ \Bigl\{\ \alpha(\G\backslash\n),\ \alpha(\G/\n)\cup\{\n\}\ \Bigr\}.$$
Also, we have that $-_\n\G$ is bipolar w.r.t. $p$ if and only if $\G\backslash\n$ and $\G/\n$ are bipolar w.r.t. $p$.
\end{cor}

\begin{proof}
Direct by Theorem \ref{thm:ind-10} and Theorem \ref{thm:bij-10} (bijection property).
\end{proof}


\begin{remark}[computational complexity]
\label{rk:difficult}
 \rm
The algorithm of Theorem \ref{thm:ind-10} is exponential time, as it may involve an exponential number of minors. Indeed, in general, one needs to compute both $\alpha(\G\backslash\n)$ and  $\alpha(\G/\n)$ in order to compute $\alpha(\G)$
(because one might compute  $T'=\alpha(\G\backslash\n)$ and finally set 
$\alpha(\G)=\alpha(\G/\n)\cup\{\n\}$, or equivalently one might compute $T''=\alpha(\G/\n)$ and finally set $\alpha(\G)=\alpha(\G\backslash\n)$).
And hence, in general, one may need to compute $\alpha(\G\backslash\n\backslash\n')$, $\alpha(\G\backslash\n/\n')$, $\alpha(\G/\n\backslash\n')$ and $\alpha(\G/\n/\n')$, with $\n'=\max(E\backslash \{\n\})$, and so on...
Finally, with $|E|=n$, the number of calls to the algorithm to build $\alpha(\G)$ is $O(2^0+2^1+\ldots +2^{n-1})$, that is $O(2^n)$.
%
In contrast, the algorithm provided in Section \ref{sec:fob} involves a linear number of minors (and yields a polynomial time algorithm, see Corollary \ref{cor:complex-alpha}). 
However, the algorithm of Theorem \ref{thm:ind-10} is efficient 
in terms of computational complexity 
for building the images of all bipolar orientations of $G$ at once,  see Remark \ref{rk:ind-10}.
\end{remark}

%

\begin{remark}[building the whole bijection at once]
\label{rk:ind-10}
 \rm
By Corollary \ref{cor:ind-10}, 
 the construction of Theorem \ref{thm:ind-10} can be used to build the whole active bijection for $G$ (i.e. the $1-1$ correspondence between all bipolar orientations of $G$ w.r.t. $p$ with fixed orientation, and all spanning trees of $G$ with internal activity $1$ and external activity $0$), from the whole active bijections for $G/\n$ and $G\bk\n$.
For each pair of bipolar orientations $\{\G, -_\n\G\}$, the algorithm  provides which ``choice'' is right to associate one orientation with the orientation induced in $G/\n$ and the other  with the orientation induced in $G\s\n$.
We mention that this ``choice'' notion is developed further in
\cite[Section \ref{ABG2-sec:induction}]{ABG2} 
(and \cite{AB4}) as the basic component for a deletion/contraction framework.
This ability to build the whole bijection at once is interesting from a structural viewpoint, but also from a computational complexity viewpoint. 
Precisely, with $|E|=n$, the number of calls to the algorithm to build one image is $O(2^n)$ (see Remark \ref{rk:difficult}), but the number of calls to the algorithm to build the $O(2^n)$ images of all bipolar orientations is  $O(2^n\times 2^0+ 2^{n-1}\times 2^1+\ldots+2^0\times 2^n+ 2^1\times 2^{n-1})$, that is just $O(n.2^n)$.
\emevder{a decreire ent ermes de "enumeration compelxity", voir courriels Daniel}%
\end{remark}

\begin{remark}[linear programming]
\label{rk:ind-lp}
 \rm
As the full optimality notion strengthens real linear programming optimality, the deletion/contraction algorithm of Theorem \ref{thm:ind-10} corresponds to a refinement of the classical linear programming solving by constraint/variable  deletion, see \cite{GiLV04, AB3} (see also Remark \ref{rk:lp}).
\end{remark}

\begin{remark}[equivalence in Theorem \ref{thm:ind-10}]%
\emevder{attention this Remark \ref{rk:ind-10-equivalence} is also repeated in the companion paper \cite{ABG2}.}
\label{rk:ind-10-equivalence}
 \rm
The equivalence of the two formulations in the algorithm of Theorem \ref{thm:ind-10} is a deep result,  difficult to prove if no property of the computed mapping is known.
Here, to prove it, we  implicitly use that $\alpha$ is already well-defined by Definition \ref{def:acyc-alpha}, 
and bijective for bipolar orientations (Key Theorem \ref{thm:bij-10}).
But if one defines a mapping $\alpha$ from scratch as in the algorithm (with either one of the two formulations) and then investigates its properties, then it turns out that
the above equivalence result is 
equivalent to the  existence and uniqueness of the fully optimal spanning tree as defined in Definition \ref{def:acyc-alpha}. 
See \cite{AB4} for precisions.
%
This equivalence result is also related to the active duality property  (see \cite[Section \ref{ABG2-subsec:fob}]{ABG2}, see also \cite[Section 5]{AB1} or \cite[Section 5]{AB2-b}\emevder{verifier ref ection AB2b}).
Roughly, in terms of graphs, from \cite[Section 4]{GiLV05}, one defines a bijection between orientations obtained from bipolar orientations by reversing 
the source-sink edge and spanning trees with internal activity $0$ and external activity~$1$.
The full optimality criterion of Definition \ref{def:acyc-alpha} can be directly adapted to these orientations with almost no change (see \cite[Section \ref{ABG2-subsec:fob}]{ABG2}).
Then, thanks to the equivalence of the two dual formulations in the above algorithm, one can directly adapt the above algorithm for this second bijection, with no risk of inconsistency.
In an oriented matroid setting, this equivalence result also means that the same algorithm can be equally used in the dual, with no risk of inconsistency.
These subtleties are detailed in \cite{AB4}.
\emevder{TOUE CETTE RK A BIEN VERIFIER, ET A VOIR AVEC - ET REPRENDRE DANS - AB4 !!!!}
\emevder{cette remark a mettre dans ABG2 ou dans ABG2LP ?????}
\end{remark}



\section{Direct computation by optimization}
\label{sec:fob}

This section gives a direct and efficient construction of the fully optimal spanning tree of an ordered bipolar digraph, in terms of an optimization based on a linear ordering of cocycles in a minor for each edge of the resulting spanning tree. 
It is completely independent of Section \ref{sec:induction}.
It is an adaptation for graphs\emevder{ai enleve a simlification car AB3 est sipple, encore lus !} of a general construction
by elaborations on linear programming given for oriented matroids in \cite{AB3} (see Section~\ref{sec:intro} for an outline, see also \cite{GiLV09} for a short statement in real hyperplane arrangements,  note that this section could be directly generalized to regular matroids or totally unimodular matrices, and see Remark \ref{rk:lp} for more information on these linear programming aspects).
By this way, the computation of the fully optimal spanning tree can be made in polynomial time (Corollary \ref{cor:complex-alpha}).
%
Also, we insist that we do not give here a new proof of the key Theorem \ref{thm:bij-10}:
we  use this 
result to assume that the fully optimal spanning tree exists, and then prove that our algorithm necessarily builds it.

\ss





\begin{lemma}
\label{lem:ordering-validity}
Let $G$ be a graph. Given  a spanning tree $T$ of $G$ and $U\subseteq T$, there exists at most one cocycle $C$ of $G$ such that $C\cap T=U$.
Given  a spanning tree $T$ of $G$ and two cocycles  $C$ and $D$ of $G$, there exists $f\in T$ belonging to $C\Delta D$.
\end{lemma}

\begin{proof}
Let us prove the first claim. Assume such a cocycle $C$ exists.
Then it is defined by a partition of the set of vertices of $G$ into two sets $A$ and $B$, such that $G[A]$ and $G[B]$ are connected (since $C$ is minimal for inclusion).
Each connected component of $T\setminus U$ is contained either in $G[A]$ or in $G[B]$.
On the other hand, if an edge $e=(x,y)$ is in $U$ then it is in $C$ and has an extremity in $G[A]$ and the other in $G[B]$. Hence, as $T$ is spanning, the partition into the sets $A$ and $B$ is completely determined by $T$ and $U$, which implies the uniqueness of $C$.
Now, let us prove the second claim.
Assume $T\cap (C\Delta D)=\emptyset$, then we have $T\cap C\subseteq D$ and $T\cap D\subseteq C$, so
$T\cap C=T\cap D= T\cap C\cap D$, which is a contradiction with the first claim for $U=T\cap C=T\cap D$.
\end{proof}

\begin{definition}
\label{def:optimizable-digraph}
\rm
We call \emph{optimizable digraph} a (connected) 
digraph $\G=(V,E\cup F)$, 
with possibly $E\cap F\not=\emptyset$,
given with:

\noindent\bul an edge $p\in E\setminus F$, called \emph{infinity edge},

\noindent\bul a subset of edges $E$, called \emph{ground set}, such that the digraph $\G(E)$ is acyclic,

\noindent\bul a subset of edges $F$, called \emph{objective set},  linearly ordered, such that $F\cup\{p\}$ is a spanning tree~of~$G$.%



\ms


Given an optimizable digraph $\G$ as defined above, 
we define a \emph{linear  ordering for the signed cocycles of $\G$ containing $p$} as follows.
Let $C$ and $D$ be two signed cocycles (see Section \ref{subsec:prelim-tools}) of $\G$ containing $p$. 
By Lemma \ref{lem:ordering-validity}, since $F\cup\{p\}$ is a spanning tree of $G$, there exists an element of $F$ that belongs to $C \Delta D$.
Let $f$ be the smallest element of $F$ such that either $f$ belongs to $C \Delta D$, or $f$ belongs to $C\cap D$ and has opposite signs in $C$ and $D$.
\emevder{j'y comprend plus rien, comment c'est possible que f ait signe differents dans C et D ? comme C et D contiennt p, c'est possible seulement si f et p coupent la ligne C-D au meme point, ce qui n'a pas d'interet...???!! voire preuve de complexite proposition plus loin, on ne se sert que de $f\in C\triangle D$... A VERIFIER !!!!! (ptet un pb de signe de p?)}%
If $f$ is positive in $C$ or negative in $D$ then set $C>D$.
If $f$ is negative in $C$ or positive in $D$ then set $D>C$.
%

Equivalently, we set $C>D$ if $f$ is a positive element of $C$ and not an element of $D$, or $f$ is a positive element of $C$ and a negative element of $D$, or $f$ is not an element of $C$ and a negative element of $D$, where $f$ is the smallest possible element of $F$ that allows for setting $C<D$ or $D<C$ by this way.

Equivalently, it is easy to see that this ordering is provided by the weight function $w$ on signed cocycles of $G$
defined as follows. 
Denote $F=f_2<...<f_r$, and denote $f_i$ resp. $\bar f_i$ the element with a positive resp. negative sign. 
For $1\leq i\leq r$, set
$w(f_i)=2^{r-i}$ and $w(\bar f_i)=-2^{r-i}$, and set $w(e)=0$ if $e\in E\setminus F$. Then define the weight $w(C)$ of a signed cocycle $C$ as the sum  of weights of its elements. The above linear ordering is 
given by:
$C>D$ if $w(C)>w(D)$.

\eme{ai change ordre dans def ci-dessus, et maximalite de optimal, et une occurence de > dans preuve, noramelment aps de probleme}

%
\ms

We define \emph{the optimal cocycle} of $\G$ as the 
maximal
signed cocycle of $\G$ containing $p$ with positive sign, and inducing a directed cocycle of $\G(E)$. It exists since $\G(E)$ is acyclic, and it is unique since the ordering is linear.

\end{definition}

\begin{thm}
\label{th:fob}
Let $\G=(V,E)$ be an ordered bipolar digraph with respect to $p=\min(E)$.
The fully optimal spanning tree $\alpha(\G)=p<t_2<...<t_r$ is computed by the following algorithm.
\medskip


\begin{algorithme}
\parindent= 1cm
\noindent(1) Initialize $\G$ as the optimizable digraph 
given by:

\hskip 2cm
\vbox{\hsize=13cm
\noindent\bul 
the infinity edge $p$ 

\noindent\bul 
the ground set $E$

\noindent\bul 
the objective set $F=f_2<...<f_r$ where $p<f_2<...<f_r$ is the smallest lexicographic spanning tree of $\G$ (and the linear ordering on $F$ is induced by the linear ordering on $E$).

}
\medskip

\noindent(2) For $i$ from 2 to $r$ do:

(2.1) Let $\C$ be the optimal cocycle of $\G$.

\noindent{\small \it Scholia: $\C$ is actually the cocycle induced in the current digraph $\G$ by the fundamental cocycle of the element $t_{i-1}$ (computed at the previous step, $t_1=p$) of the fully optimal spanning tree of the initial digraph.}%
\emevder{peut-tre faire un vrai enonce pour ca, attention c'est dit pluieur fois, dans exemple aussi et dns preuve de observation a la fin aussi... et dans preuve du thm propriete P2... c'est important}%
\emevder{+++ en fait serait ptet bien d'observer avant  que cnnaiter bse interne = connaitre fundamental cocycles}%
\ss

(2.2) Let $$t_i=\min (E\setminus \C)$$

\noindent{\small \it Scholia: the $i$-th edge $t_i$ of $\alpha(\G)$ is actually the smallest edge not contained in fundamental cocycles of smaller edges of $\alpha(\G)$ because the fully optimal spanning tree $\alpha(\G)$ is internal uniactive.}%
\ss

(2.3) Let $$p'=t_i$$

(2.4) Let $$E'=E\setminus \C$$

(2.5) %
\vtop{\hsize=13cm
\noindent If $p'\in F$ then let $F'=F\setminus\{p'\}$, and if $p'\not\in F$ then let: }
$$ F'= F\setminus \max \bigl(\ F\cap  C_{G} (F \cup \{p\};p')\ \bigr).$$

\hskip 1cm%
\vtop{\hsize=14cm
\noindent 
Equivalently,
$F'$ is obtained  by removing from $F$ 
the greatest possible element such that the following property holds: 

\it $p'\not\in F'$ and $F'\cup\{p'\}$ is a spanning tree of the minor $\G'$ defined below. }
\ss

%

(2.6) Set
$$\G'= \G\ (\ E'\ \cup\  F'\ \cup\ \{p\}\ )\ /\ \{p\}$$

\hskip 1cm
as the optimizable digraph 
given by:

\hskip 2cm
\vbox{\hsize=12cm

\noindent\bul 
the infinity edge $p'$ 

\noindent\bul 
the ground set $E'$

\noindent\bul 
the objective set $F'$

}

(2.7) Update $\G:=\G'; \ p:=p'; \ E:=E'; \ F:=F'.$
\medskip

\noindent(3) Output $$\alpha(\G)=p<t_2<...<t_r.$$
\end{algorithme}
\end{thm}

\vspace{-10mm}

\begin{figure}[H]
	\centering
\includegraphics[scale=0.42]{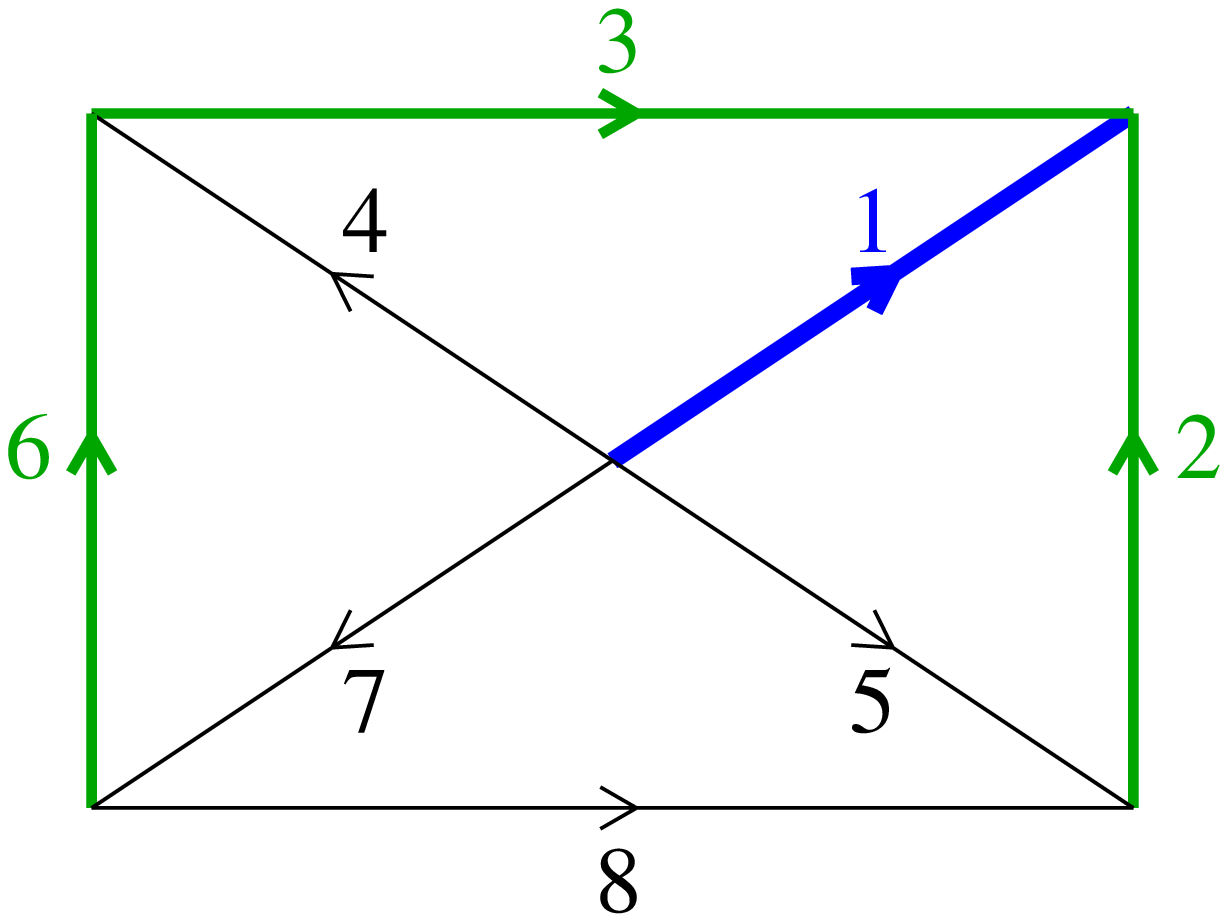}
\hfill
			\includegraphics[scale=0.42]{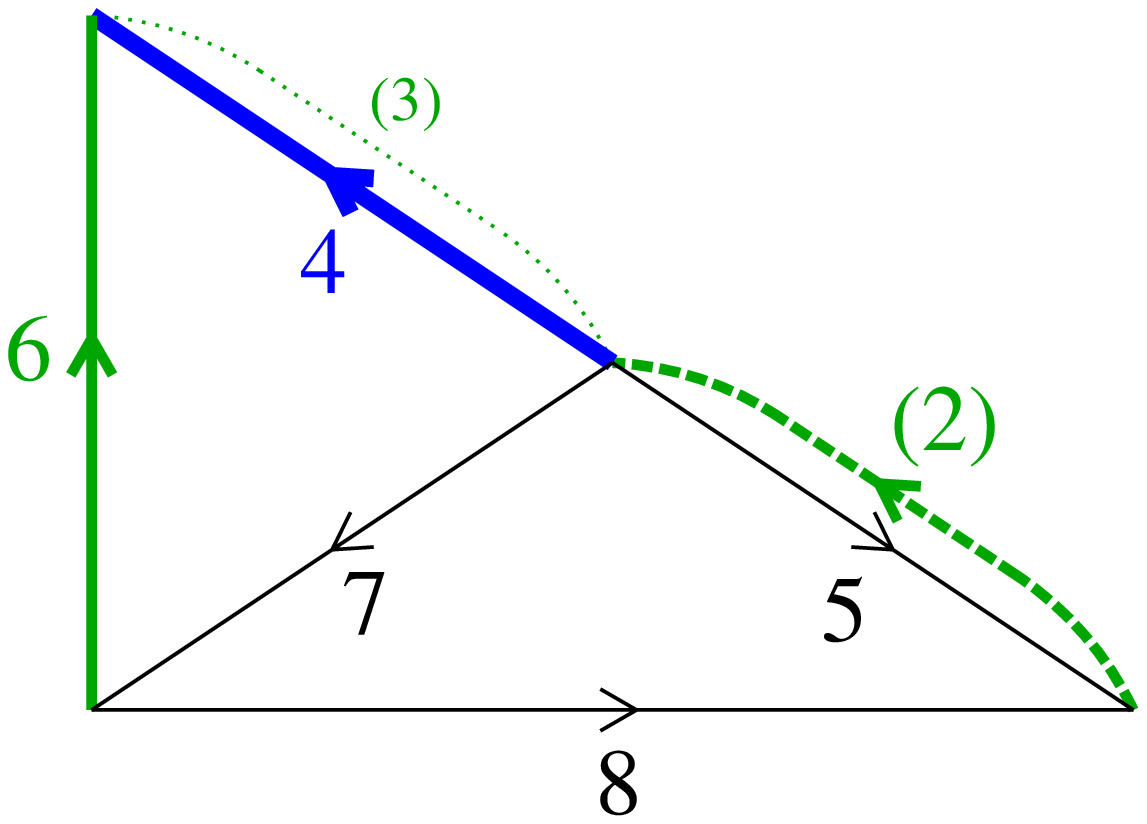}	
		\hfill
		\includegraphics[scale=0.42]{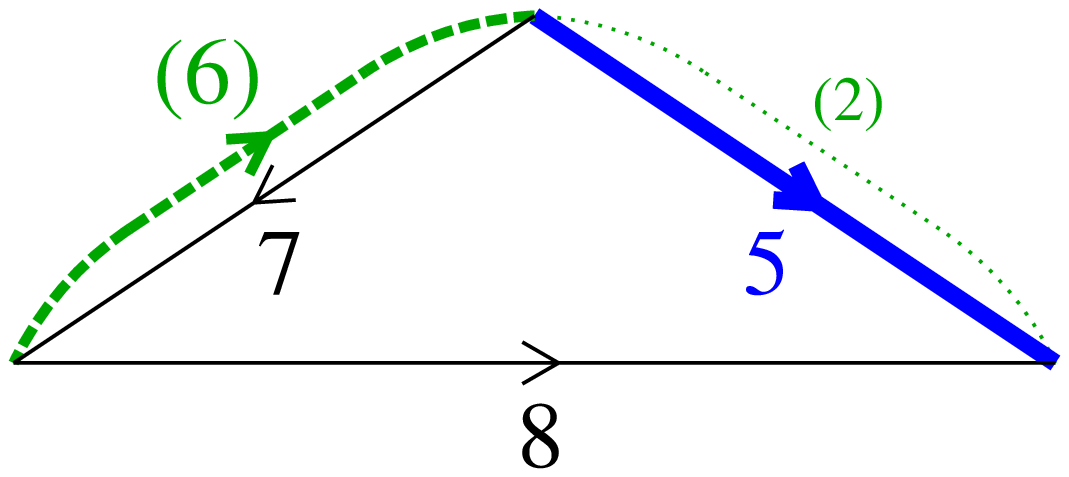}
		
		\caption[]{illustration for Example \ref{example}. The digraph $\G=\G_1$ is on the left, the digraph $\G_2=\G_1/1\backslash 3$ in the middle, and the digraph $\G_3=\G_2/4\backslash 2$ on the right. At each step: the bolder edge is $p$, the other bold edges form the set $F$, the directed edges within parenthesis are in $F$ but no more in $E$, and the nearly invisible dashed edge is not part of the graph since it has been deleted from the previous $F$.}
		\label{fig:fob-algo}
\end{figure}

\vspace{-3mm}

\begin{figure}[H]
	\centering
		\includegraphics[scale=0.4]{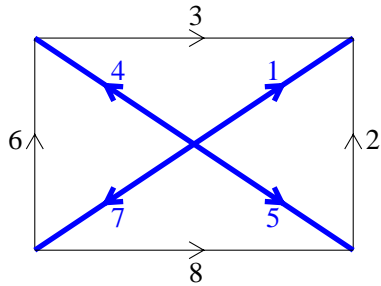}
		\caption[]{the fully optimal spanning tree  $\alpha(\G)=1457$ of the digraph $\G$ of Figure \ref{fig:fob-algo} (the same as in \cite{GiLV05}).}
		\label{fig:fob-output}
\end{figure}

\vspace{-3mm}

\begin{example}
\label{example}
\rm
Theorem \ref{th:fob} might seem rather technical in a pure graph setting, so let us first illustrate it on an example before proving it  (see Remark \ref{rk:lp} or \cite{GiLV09, AB3} for geometrical interpretations).
Let us apply the algorithm of Theorem \ref{th:fob} to the same example as in \cite{GiLV05}, where it illustrated the inverse algorithm. The steps of the algorithm are depicted on Figure \ref{fig:fob-algo}. The output is depicted on Figure~\ref{fig:fob-output}.
One can keep in mind that the successive optimal cocycles built in the algorithm correspond to the 
fundamental cocycles with respect to the successive edges of the fully optimal spanning tree (cf. scholia above, see details in property (P2$_i$) in the proof of Theorem \ref{th:fob} below).\emevder{deja dit dans scholia !!!}%
\smallskip

{
\parindent=1cm

\noindent {\it --- Computation of $t_2$.}
Initially $\G=\G_1$ is a digraph with set of edges $E= \{1<2<3<4<5<6<7<8\}$.
The  minimal spanning tree is $\{1<2<3<6\}$.
Hence $p = 1$ and 
$F = 236$.
%
%
The linearly ordered directed cocycles of $\G_1$ containing $p=1$ are:
%
$123\ > \ 1246\ > \ 1358\ > \ 14568\ > \ 1457$.
%
The maximal is $\C= 123$ (which is equal, at this first step, to the smallest for the lexicographic ordering, see Observation \ref{obs:first-cocycle} below).
%
%
%
We get $t_2=4$.
\smallskip


\noindent {\it --- Computation of $t_3$.}
%
We now consider the optimizable digraph $\G_2=\G_1 / \{1\} \backslash \{3\}$
%
given with
%
$p = 4$,
$E = 45678$,
and
$F = 26$ 
(the edge $3$ is deleted as the greatest belonging to the circuit $134$ and to the previous $F=236$).
%
The linearly ordered signed cocycles of $\G_2$ positive of $E$ and containing $p=4$ are:
%
$46\ >\ \overline{2}457 $
%
(where the bar upon elements represents negative elements).
%
%
\noindent The maximal is $\C=46$.
We get $t_3=5$.

\smallskip

\noindent {\it --- Computation of $t_4$.}
We now consider the  optimizable digraph  $\G_3 = \G_2 / \{4\} \backslash \{ 2\}$
given with
$p=5$,
$E=578$,
and
$F=6$
%
(the edge $2$ is deleted as the greatest belonging to the circuit $25$ and to the previous $F=26$).
The linearly ordered signed cocycles of $\G_3$ positive of $E$ and containing $p=5$ are:
%
$58\ >\ 5\overline{6}7$.
%
%
%
The maximal is $\C=58$.
%
We get $t_4=7$.
\smallskip

\noindent {\it --- Output.}
We get finally $\alpha(\G)=1457$
(and one can check that the fundamental cocycles $C^*(1457;1)=123$, $C^*(1457;4)=\overline{3}46$, and $C^*(1457;5)=\overline{2}58$ induce the successive optimal cocycles $123$, $46$ and $58$ in the successive considered minors).
}
\end{example}

\begin{paragraph}{\textit{\textbf{Notations for what follows}}}
The proof of Theorem \ref{th:fob} is given below, after two lemmas. 
In all this framework, we will use the following notations.
We denote $\G$, $\G'$, etc., the variables as they are used during the algorithm, meaning they are considered at any given step of the algorithm with their current value.
We denote $\G_1$ the initial  optimizable digraph $\G$, and $\G_i$ the variable  optimizable digraph $\G$ updated at step (2.7) w.r.t. variable $i$, with parameters $p_i=t_i$ as $p$, $E_i$ as $E$ and $F_i$ as $F$, for all
$1\leq i\leq r$.
\end{paragraph}

\begin{lemma}
\label{lem:algo-fob-well-defined}
The algorithm of Theorem \ref{th:fob} is well-defined.
\end{lemma}

\begin{proof}
%
Initially, the optimizable digraph $\G=\G_1$  is obviously well defined.
One has to check that the optimizable digraph $\G'$ defined at each call to step (2.6) is well defined.
%
First, by induction hypothesis, $\G$ is connected, and $\C$ is a cocycle of $\G$, that is an inclusion-minimal cut. We have $p\in \C$ by definition of $\C$, and we have $E'=E\setminus \C$ by step (2.4), hence $p$ is the unique edge joining the two connected components of $\G(E'\cup \{p\})$.
Hence $\G(E'\cup \{p\})/\{p\}$ is connected, and so is $\G'$.

By definitions at steps (2.2)(2.3)(2.4), we have $p'=\min(E')$, and by definition of $F'$ at step (2.5), we have $p'\not\in F$, hence $p'\in E'\setminus F'$, as required for an optimizable digraph.
Since $F\cup\{p\}$ is a spanning tree of $\G$, then $F$ is a spanning tree of $\G'$, and, by definition of $F'$ at step (2.5),  $F'\cup\{p'\}$ is a spanning tree of $\G'$, as required for an optimizable digraph.
By assumption, at each call to step (2.6), the digraph $\G(E)$  is acyclic.
Since $\C$ is a cocycle of $\G$ with $p\in \C$ as above, then $\G(E'\cup \{p\})$ is also acyclic, and so is $\G'(E')$, as required for an optimizable digraph.
\end{proof}


\begin{lemma}
\label{lem:algo-fob-invariant}
%
At any step of the algorithm of Theorem \ref{th:fob}, 
the following invariant is maintained:
the smallest element of a cocycle of $G$ belongs to $F\cup\{p\}$, and this element is equal to the smallest element of the cocycle of $G_1$ inducing this cocycle (in the minor $G$ of $G_1$, see Section \ref{subsec:prelim-tools}).
\end{lemma}

\begin{proof}
The property is true at the initial step since $F\cup\{p\}=F_1\cup\{p_1\}$ is defined as the smallest spanning tree of $G=G_1$.
Assume it is true for $G$, we want it true also for $G'$ as defined at step (2.6).
Let $D'$ be a cocycle of $G'$. By construction of $G'=G\backslash A/\{p\}$ for some set $A$ such that $G'$ is connected, there exists a cocycle $D$ of $G$ inducing $D'$, such that $p\not\in D$ and $D\setminus A=D'$ (see Section \ref{subsec:prelim-tools}). 
Since $F\cap A=F\setminus (E'\cup F' \cup \{p\})$ by definition of $G'$, and $F\setminus F'$ contains exactly one element $f$ by definition of $F'$ at step (2.5), then $F\cap A$ contains at most this element $f$.
%
By induction hypothesis, we have $\min(D)\in F$. By definition of $D$, we have $D\setminus A=D'$.
%
Assume for a contradiction that $\min(D')\not= \min(D)$.
Then
$\min(D)\in A\cap F$, implying $F\cap A=\{f\}$ and $\min(D)=f$.
There are two cases for defining $f$ at step (2.5).
In the first case, we have $f=p'=\min(E')\in F$, which implies $f\in E'$ and which contradicts $\{f\}=F\cap A=F\setminus (E'\cup F' \cup \{p\})$. 
In the second case, $f$ is defined as the greatest element of the unique cycle $C$ of $G$ contained in $F\cup\{p,p'\}$. 
In this case, assume an element $f'$ belongs to $D\cap C$. 
We have $f'\leq f$ since $f'\in C$ and the greatest element of $C$ if $f$. 
And we have $f'\geq f$ since $f'\in D$ and $\min(D)=f$. 
Hence $f'=f$, and we have proved $D\cap C=\{f\}$,
 which contradicts the orthogonality of the cycle $C$ and the cocycle $D$ (see Section \ref{subsec:prelim-tools}).

So we have $\min(D')= \min(D)$.
Since the cocycle $D_1$ of $G_1$ inducing $D$ in $G$ also induces $D'$ in $G'$, and since $\min(D)=\min(D_1)$ by induction hypothesis, we get $\min(D')=\min(D_1)$.

Finally, if $\min(D')\not\in F'$, since $\min(D')=\min(D)\in F$, then $\min(D')=f$. As above, there are two cases for defining $f$ at step (2.5).
In the first case, we have $f=p'=\min(E')$, hence $\min(D')=f\in F'\cup\{p'\}$, which is the property that we have to prove.
In the second case, we have $\min(D)=f$ and the same argument as above leads again to  $C\cap D=\{f\}$ and to the same contradiction.
 The invariant is now proved.
 %
 %
\end{proof}


\emevder{IMPORTANT attention : voir ou est-ce qu'on utilise que G bipolar au depart !!!! en fait c'est utilis? par le fait qu'il existe $T$ avec ses proprietes !!! vu vite fait avant soumission, mais a revoir apres redaction de AB3 !!! cf. phrase au debut dans la preuve --- OK C'EST BON JE CROIS}

\begin{proof}[Proof of Theorem \ref{th:fob}] 
The present proof is a condensed but complete version for graphs
of the 
general geometrical proof that will be given in \cite{AB3}, taking benefit of the linear ordering of cocycles defined above (which is possible in graphs but not in general, see Remark \ref{rk:lp}).
%
We will extensively use the notion of cocycle induced in a minor of a graph by a cocycle of this graph, see  Section \ref{subsec:prelim-tools}.
%
%
%
Since $\G_1$ is bipolar, its fully optimal spanning tree $\alpha(\G_1)$ 
exists
and satisfies the properties recalled in Section \ref{subsec:prelim-fob}. 
By Lemma \ref{lem:algo-fob-well-defined}, the algorithm  given in the theorem statement is well defined.
Now, we have to prove that $\alpha(\G_1)$ is necessarily equal to the output of this algorithm. 
%
%
The proof consists in proving by induction that, for every $2\leq i\leq r$, the two following properties (P1$_i$) and (P2$_i$) are true. The property (P1$_i$) for $2\leq i\leq r$ means that the algorithm actually returns $\alpha(\G_1)$.
The property (P2$_i$) for $2\leq i\leq r$ means that the optimal cocycles computed in the successive minors are actually induced by the fundamental cocycles of the fully optimal spanning tree $\alpha(\G_1)$, as noted in the first scholia in the algorithm statement. Then, the proof is not long but rather technical. 
Let us denote $T=\alpha(\G_1)=\{b_1<\ldots<b_r\}$.
\begin{itemize}
\partopsep=0mm \topsep=0mm \parsep=0mm \itemsep=0mm
\item 
\noindent (P1$_i$)  {\it We have $b_i=t_i$, where $t_i$ is defined at step (2.2) of the algorithm 
and $b_i$ is the $i$-th element of $T=\alpha(\G_1)$.}
\item
\noindent (P2$_i$) {\it The optimal cocycle $\C$ of $\G_{i-1}$, defined at step (2.1) of the algorithm 
equals the cocycle denoted $C_{i-1}$ of $\G_{i-1}$ induced by the fundamental cocycle $C^*(T;b_{i-1})$ of $p_{i-1}=t_{i-1}=b_{i-1}$ w.r.t. $T$ in~$\G_1$,
that is: $C_{i-1}= C^*(T;b_{i-1}) \cap (E_{i-1}\cup F_{i-1})$, where $(E_{i-1}\cup F_{i-1})$ is the edge set of $\G_{i-1}$.%
}
\end{itemize}

\eme{bien relire cette preuve}%
First, observe that $C_{i-1}$ is a well defined induced cocycle in property (P2$_i$) as soon as (P1$_j$) is true for all $j<i-1$ (implying $t_j=b_j$), since $b_j\not\in C^*(T;b_{i-1})$ for $j<i-1$ and $\G_{i-1}=\G_1/\{p_1,t_2,...,t_{i-2}\}\setminus A$ for some $A$.

Second, let $1\leq k\leq r$.
Assume that, the property (P2$_i$) is true for all $2\leq i\leq k$
and the property (P1$_i$) is true for all $2\leq i\leq k-1$.
Then we directly have that the property (P1$_i$) is also true for $i=k$. 
Indeed, as shown in \cite[Proposition 2]{GiLV05} (that can be proved easily), the fact that the spanning tree $T=b_1<b_2<...<b_r$ 
 has  external activity $0$
  implies that,  for all $1\leq i\leq r$, we have $b_i=\min \bigl(E\setminus \cup_{j\leq i-1} C^*(T;b_j)\bigr)$.
So we have
$b_{k}=\min(E_{k-1} \setminus C_{k-1})=\dots=\min(E_1\setminus \cup_{j\leq k-1} C_j)=\min(E_1\setminus \cup_{j\leq k-1} C^*(T;b_j))$, and so we have $t_{k}=b_{k}$.

Now, it remains to prove that, under the above assumption, the property (P2$_i$) is true for $i=k+1$.
%
Consider $\C$, denoting the optimal cocycle of $\G_k$, and $C_k$, denoting the cocycle  of $\G_{k}$ induced by the fundamental cocycle of $p_k=t_{k}=b_{k}$ w.r.t. $T$ in~$\G_1$. 
Assume for a contradiction that $\C\not=C_k$. 
\smallskip

Since properties (P1$_i$) and (P2$_i$) are true for all $2\leq i\leq k$ by assumption, and reformulating definition of $\G'$ at step (2.6), we have: 
\begin{eqnarray*}
\G_k&=&\G_{k-1}/\{b_{k-1}\}\setminus \Bigl(C_{k-1}\setminus (F_k \cup \{b_{k-1}\})\Bigr)\\
&=&\G_{k-1}/\{b_{k-1}\}\setminus \Bigl(C^*(T;b_{k-1})\setminus (F_k \cup \{b_{k-1}\})\Bigr)
\end{eqnarray*}
and hence, inductively, we have:
$$\G_k=\G_1/\{b_1,...,b_{k-1}\}\backslash A$$ where $A$ is the union of all fundamental cocycles of $b_i$ w.r.t. $T$ for $1\leq i\leq k-1$ minus $F_k \cup \{b_1,...,b_{k-1}\}$. That is:
$$A=\Bigl(\ \bigcup_{1\leq i\leq k-1}C^*(T;b_i)\ \Bigr) \ \ \setminus\ \  \Bigl(\ F_k \cup \{b_1,...,b_{k-1}\}\ \Bigr).$$
In particular, we have $C_k=C^*(T;b_k)\setminus A$.
And we also have $A\cap T=\emptyset$. 

As recalled in Section \ref{subsec:prelim-fob},
the definition of $T$ implies that $C_k$ is positive on $E_k$, that is positive except maybe on elements of $F_k\setminus E_k$, just the same as $\C$.
By assumption and property (P1), we have $b_k=t_k=\min(E_k)$.
Then, by definition of $C_k= C^*(T;b_{k}) \cap (E_{k}\cup F_{k})$, we have $b_k\in C_k$. 
Hence, the cocycle $C_k$ has been taken into account in the linear ordering of cocycles of the optimizable digraph $\G_k$ defining $\C$. So $\C>C_k$ in this linear ordering by definition of $\C$. 
\smallskip

By definition of this linear ordering, let $f$ be the smallest element of $F_k$ with the property of being positive in $\C$ and not belonging to $C_k$,
or positive in $\C$ and negative in $C_k$,
or not belonging to $\C$ and negative in $C_k$.
The edge $f$ does not have opposite signs in $\C$ and $-C_k$.
So let $D'$ be a cocycle of $\G_k$ obtained by elimination between $\C$ and $-C_k$ preserving $f$ (see Section \ref{subsec:prelim-tools}).
Necessarily, $f$ is positive in $D'$ by definition of $f$.
By definitions of $\C$ and $C_k$, the element $b_k$ is positive in $\C$ and in $C_k$. 
Moreover, every edge in $F$ smaller than $f$ belonging to $\C$ resp. $C_k$ also belongs to $C_k$ resp. $\C$ with the same sign, by definition of $f$.
Hence, by properties of elimination, the cocycle $D'$ does not contain $b_k$ nor any element of $F$ smaller than $f$ belonging to $\C$ or $C_k$.
Since the smallest edge in $D'$ belongs to $F\cup\{b_k\}$, as shown by the invariant of Lemma \ref{lem:algo-fob-invariant}, and since we have $b_k\not\in D'$, then we have $\min(D')=f$.
The negative elements of $D'$ are either elements of $F_k\setminus E_k$, or elements of $C_k\setminus \{b_k\}$.
In the first case,  since $E_k=E_1\setminus A$, the negative elements belong to $F_k\cap A\subseteq A\subseteq E_1\setminus T$.
In the second case, the negative elements also belong to $E_1\setminus T$
 by definition of a fundamental cocycle.
Finally, let $D$ be the cocycle of $\G_1$ inducing $D'$ in $\G_k$,
such that $D\cap \{b_1,...,b_k\}=\emptyset$ and $D\setminus A=D'$.
The negative elements of $D$ belong to $A$ or are negative in $D'$, then, in every case, they belong to $E_1\setminus T$.
Moreover, as shown by the invariant of Lemma \ref{lem:algo-fob-invariant}, we have $\min(D)=\min(D')=f$.

So we have built a cocycle $D$ such that:

(i) $D\cap \{b_1,...,b_k\}=\emptyset$

(ii) $\min(D)=f$ 

(iii) $f$ is positive in $D$

(iv) the negative elements of $D$ are in $E_1\setminus T$
\smallskip

In a first case, we assume that $f\not\in T$. Then $f=c_j$ for some $c_j\in E_1\setminus T$. 
Let $C=C(T;c_1)\circ ...\circ C(T;c_j)$. As recalled above, this composition of cycles has only positive elements except the first one $p_1=b_1$.
The edge $f$ is positive in $C$ and $D$, hence, by orthogonality (see Section \ref{subsec:prelim-tools}), there exists an edge $e\in C\cap D$ with opposite signs in $C$ and $D$.
We have $b_1\not\in D$, hence $e\not=b_1$, and hence $e$ is positive in $C$ and negative in $D$. Hence $e\in E_1\setminus T$. Since $e\in C$, we must have $e=c_i$ for some $i\leq j$, that is $e \leq f$. 
But $f=\min(D)$ implies $e=f$, which is a contradiction.
\smallskip
 
In a second case, we assume that $f\in T$. Let $a=\min \bigl(C^*(T;f)\bigr)$.
Since $f\in F$, we have $f\not=\min(E_1)$, and so $f\not= a$ by definition of $T$, so we have $a\not\in T$. Then $a=c_j$ for some $c_j\in E_1\setminus T$. 
Let $C=C(T;c_1)\circ ...\circ C(T;c_j)$, which has only positive elements except the first one $p_1=b_1$. 
Since $a\in C^*(T;f)$, we have $f\in C(T;c_j)$.
As above, by orthogonality, the edge $f$ positive in $C$ and $D$, together with $p_1\not\in D$, implies an edge $e\in C\cap D$ positive in $C$ and negative in $D$. So $e\in E_1\setminus T$, and so $e=c_i$ for some $i\leq j$,
so $e\leq a$, which implies $e\leq f$ by definition of $a$, 
leading to the same contradiction as above with $f=\min(D)$.
\end{proof}

\emevder{mettre en corollaire/lemme/scolie la pty P2$_i$ selon laqulle les cocircuits optimaux sont les cocircuits fondamentaux de base ?}

\begin{remark}[linear programming]
\label{rk:lp}
\rm

The algorithm of Theorem \ref{th:fob} actually consists in a translation and an adaptation in the case of graphs of a geometrical algorithm providing in general  a strengthening of pseudo/real linear programming (in oriented matroids / real hyperplane arrangements).  
In this setting, we optimize a sequence of nested faces (the successive optimal cocycles in the above algorithm, a process that we call flag programming), each with respect to a sequence of objective functions (the linearly ordered objective set in the above algorithm, a process that we call multiobjective programming), yielding a unique fully optimal basis for any bounded region. This refines standard linear programming where just one vertex is optimized with respect to just one objective function,
but this can be computed inductively using standard pseudo/real linear programming.
See  \cite{AB3} for details (see also \cite{GiLV09} for a short presentation   and formulation of the general algorithm in the real case; see also \cite{AB1} for the primary relations between full optimality and linear programming; see also \cite{GiLV04} in the simple case of the general position/uniform oriented matroids where the first fundamental cocircuit determines the basis; see also Section \ref{sec:intro} for complementary information and references notably on duality properties; see also Remark \ref{rk:ind-lp} in relation with deletion/contraction).

Let us relate that to the present paper. The flag programming can be addressed inductively by means of a sequence of multiobjective programming in minors. This induction is rather similar in the graph case, yielding the successive minors addressed in the above algorithm of Theorem \ref{th:fob}.
The multiobjective programming can be addressed inductively
by means of a sequence of standard linear programming in lower dimensional regions.
In the graph case, this induction can be avoided and transformed into the linear ordering of cocycles by means of a suitable weight function (Definition \ref{def:optimizable-digraph}), yielding a unique optimal cocycle. 
This simplification is possible because $U_{2,4}$ is an excluded minor (as for binary matroids), which yields very special properties for the skeleton of regions. Hence, the construction of this section could be readily extended to binary oriented matroids (that is regular oriented matroids, or totally unimodular matrices). See the proof of Proposition \ref{prop:complex-optim} for technical explanations.

From the computational complexity viewpoint, we get from this approach that  the optimal cocycle, and hence the fully optimal spanning tree, can be computed in polynomial time, a result which we state in Proposition \ref{prop:complex-optim} and Corollary \ref{cor:complex-alpha} below.
However, this complexity bound is achieved here by using numerical methods for standard real linear programming. 
An interesting question is to achieve this bound by staying at the combinatorial level of the graph, see Remark~\ref{rk:graph-complexity}.
\end{remark}

\begin{prop}
\label{prop:complex-optim}
The optimal cocycle of an optimizable digraph can be computed in polynomial time.
\end{prop}

\begin{proof}
\emevder{VITE FAITE JUSTE AVAT RESOUMISSION, A RELIRE BIEN !!! VOIRE A COMPARER AVEC AB3 oua  reprednre dans AB3!}
This proof is based on geometry and linear programming.
\emevder{ai enleve ca:, and we present it in a fast way.}%
It can be seen as a special case of a general construction in terms of multiobjective programming detailed in \cite{AB3}, see Remark \ref{rk:lp}.
We assume that the reader is familiar with the classical representation of a directed graph as a signed real hyperplane arrangement 
(which is usual in terms of oriented matroids,
see for instance 
\cite[Chapter 1]{OM99}).
We thus freely switch between graphical and geometrical terminologies.

Let us consider an optimizable digraph $\G=(V,E\cup F)$ as in Definition \ref{def:optimizable-digraph}. 
Each edge of $\G$ is identified to a hyperplane in a real vectorial space providing a positive halfspace delimited by this hyperplane. 
Since $\G(E)$ is acyclic, the intersection of positive halfspaces of hyperplanes in $E$ is a region $R$ of the space.
Concisely, the region defined by $E$ is where the optimization takes place, 
the hyperplanes in $F$ are the kernels of linear forms whose values are to be optimized
(intuitively increasing from the negative halfspace to the positive halfspace), 
and the hyperplane $p$ is considered as a hyperplane at infinity.
Precisely, let us now proceed to the affine hyperplane arrangement induced by intersections of hyperplanes with an hyperplane parallel to $p$ on the positive side of $p$.
In this affine representation, the vertices (0-dimensional faces) correspond to cocycles of $G$ containing $p$, and the region $R$ induces a region which we denote $R$ again, and the vertices of the skeleton of $R$ correspond to the directed cocycles of $\G(E)$ containing $p$.

Let us consider two adjacent vertices $v_C$ and $v_D$ in the skeleton of $R$.
 Let $L$ be the line formed by $v_C$ and $v_D$.
Let $f\in F$ be an affine hyperplane which is not parallel to the line $L$, and let $v_f$ be the intersection of $L$ and $f$.
Since the initial hyperplane arrangement has been built from a graph, the uniform matroid $U_{2,4}$ is an excluded minor of the underlying matroid, 
which implies that a line in the considered affine hyperplane arrangement has at most two intersections with hyperplanes.
Hence, among the three vertices $v_C$, $v_D$ and $v_f$ of $L$, at least two of them are equal, hence $v_f=v_C$ or $v_f=v_D$, which implies  $f\in C\triangle D$.
Note that this is where we use the fact the oriented matroid is graphical, and this is why the construction of this paper directly generalizes to regular matroids.\emevder{see \cite{AB3} ?????}\emevder{citer rk numerotee?}%

Now let us apply the definition of the ordering of cocycles of the optimizable digraph $\G$. 
Assume that $f$ is the smallest element of $F$ belonging  to $C\triangle D$.
\emevder{remark that $f$ cannot belong to $C\cap D$, hence this is actually useless in the definition when consdiering cocycles containing $p$ !!! a VOIR !!!}%
We have $C>D$ if $f$ is positive in $C$ (hence with $f\not\in D$) or negative in $D$ (hence with $f\not\in C$).
In any case,  the ordering (either $C>D$ or $D<C$) can be seen as indicating if $v_C$ has either a bigger or a smaller value than $v_D$ with respect to a linear form defined by $f$.
So, in the same manner, as shown by the weight function defining the ordering, the optimal cocycle $\C$ of $\G$ can be obtained the following way, denoting $F=f_2<...<f_r$:
first find the optimal vertices in the region $R$ with respect to a linear form defined by $f_2$ (they form a face parallel to $f_2$),
second, among these vertices, find the optimal vertices  with respect to a linear form defined by $f_3$ (they form a face parallel to $f_2\cap f_3$), and so on, until finding the uniquely determined cocycle $\C$. Note that  this multiobjective programming is part of the general construction  addressed in \cite{AB3}, which we could translate here into a linear ordering of cocycles thanks to the special geometry of a graphical arrangement.

Finding an optimal vertex with respect to a linear form can be done in polynomial time using numerical methods of real linear programming.
Hence finding the $r-1$ successive sets of optimal vertices, until the unique final one corresponding to $\C$, can also be done in polynomial time.
\end{proof}

\begin{cor}
\label{cor:complex-alpha}
The fully optimal spanning tree of an ordered directed graph which is bipolar w.r.t. its smallest edge can be computed in polynomial time.
\end{cor}

\begin{proof}
The algorithm of Theorem \ref{th:fob} consists in finding the optimal cocycle of $r-1$ successive optimizable digraphs,
built as simple minors 
of the initial digraph. 
So, we apply Proposition~\ref{prop:complex-optim}.
\end{proof}

\begin{remark}[computational complexity using a pure graph setting]
\label{rk:graph-complexity}
\rm
An interesting question is to prove Proposition \ref{prop:complex-optim} without using a numerical linear programming method, that is to build the optimal cocycle of an optimizable digraph in polynomial time 
while staying at the graph level.

Let us give an answer which is available for the first computed optimal cocycle in Theorem \ref{th:fob}, which is the fundamental cocycle of $p=\min(E)$ w.r.t. the fully optimal spanning tree of $\G$. 
For the initial optimizable digraph $\G$, the ground set is the whole edge set $E$, hence  the optimal cocycle $\C$ is a directed cocycle of $\G(E)=\G$.
In this case, no negative element has to be taken into account when defining $\C$ from the weight function of Definition \ref{def:optimizable-digraph}. Actually, one can thus also deduce a stronger property for this first $\C$ (using the fact that the objective set is built from the smallest lexicographic spanning tree), see Observation \ref{obs:first-cocycle}.

What is important is that, in this case, $\C$ turns out to be the directed cocycle of maximal weight in a bipolar acyclic digraph $\G$ for a certain weight function on (undirected) edges. 

In such a situation, in order to build such a $\C$ cocycle, one can use the celebrated \emph{Max-flow-Min-cut  Theorem} of digraphs. We refer the reader to \cite{BaGu02} for  details 
(see also for instance \cite{Sc03} 
on the problem of finding a minimum directed cut in other terms).
\emevder{voir aussi 8.4.28 page 251 de Groetschel Lovasz Schriver 1983 sur minimum weighted dicut ???}%
Roughly, start with the acyclic digraph with weights on edges, and add all opposite edges with infinite weights. By this theorem, computing a minimal cut is equivalent to computing a maximum flow, hence it can be done in polynomial time (and it is even simpler in our case where there is only one  adjacent source and sink).
Since the resulting minimum cut has a finite weight, then it  necessarily corresponds to a directed cocycle of the initial digraph (removing edges with an infinite weight), that is to $\C$.
\emevder{je n'ai rien compris a ce que j'ai ecrit, c'est d'apres note vite faite avec stephane, pas plus claire...}

%

However, this construction can not be  directly applied to compute the optimal cocycle of a general optimizable digraph (nor to compute the next fundamental cocycles of the fully optimal spanning tree), since 
the optimal cocycle is not in general a directed cocycle of the optimizable digraph
(only of its restriction to the ground set $E$) and weights of  edges in $F$  may have to be counted negatively depending on their direction.
We leave this open question for further research.
\end{remark}

\begin{observation}
\label{obs:first-cocycle}
Let $\G=(V,E)$ be an ordered bipolar digraph with respect to $p=\min(E)$.
The fundamental cocycle of $p$ w.r.t. the fully optimal spanning tree of $\G$, that is $C^*(\alpha(\G);p)$, is actually the smallest lexicographic directed cocycle of $\G$. 
We leave the details (see Remark \ref{rk:graph-complexity}).
\emevder{a mettre ou pas ? si ou a prouver meiux !}%
\end{observation}

\emevder{dessous dans source tentaive de preuve, reperete rk plus haut... plus laisse tomber... A FAIRE, je ne suis plus bien sur de ca !!!!!!}

\noindent{\bf Acknowledgments.}
\noindent Emeric Gioan 
wishes to thank J{\o}rgen Bang-Jensen and St\'ephane Bessy for communicating reference \cite{BaGu02} and  how to  use the
\emph{Max-flow-Min-cut Theorem} in Remark~\ref{rk:graph-complexity}.





\vspace{-0.2cm}

\bibliographystyle{amsplain}



\emevder{verifier refs mise a jour + faire fichier .bib ?}

\providecommand{\bysame}{\leavevmode\hbox to3em{\hrulefill}\thinspace}
\providecommand{\MR}{\relax\ifhmode\unskip\space\fi MR }
\providecommand{\href}[2]{#2}


\end{document}